\documentclass[12pt,reqno]{amsart}
\usepackage{amsmath,amssymb,amsfonts,amsthm,amstext,bbm}
\usepackage{url,xspace}
\usepackage[hypertex]{hyperref}
\usepackage{graphicx,color}

\newcommand{\Z}{{\mathbb Z}}

\newcommand{\N}{{\mathbb N}}
\renewcommand{\P}{{\mathbb P}}
\newcommand{\E}{{\mathbb E}}
\newcommand{\TT}{{\mathbb T}}
\newcommand{\BB}{{\mathbb B}}
\newcommand{\ind}{\mathbbm{1}}
\newcommand{\eps}{\epsilon}

\newcommand{\ex}{\mathfrak{X}}
\newcommand{\Esc}{E}

\newcommand{\rt}{o}
\newcommand{\nx}{\iota}
\newcommand{\trans}{\mathcal{E}}
\newcommand{\down}{\mathbf{0}}
\newcommand{\bb}{\mathbf{b}}

\newcommand{\df}{\textbf}

\newtheorem{thm}{Theorem}
\newtheorem{lemma}[thm]{Lemma}
\newtheorem{prop}[thm]{Proposition}
\newtheorem{cor}[thm]{Corollary}

\newcounter{mycount}
\newenvironment{mylist}{\begin{list}{{\rm (\roman{mycount})}}%
{\usecounter{mycount}\itemsep 0pt}}{\end{list}}

\title{Rotor walks on general trees}

\author[Angel]{Omer Angel}
\address[O.\ Angel]{\sloppypar Department of
Mathematics, University of British Columbia, Vancouver, BC V6T 1Z2, Canada}
\email{angel at math.ubc.ca}
\urladdr{http://www.math.ubc.ca/$\sim$angel}

\author[Holroyd]{Alexander E.\ Holroyd}
\address[A.\ E.\ Holroyd]{\sloppypar Microsoft Research,
1 Microsoft Way, Redmond, WA 98052, USA}
\email{holroyd at microsoft.com}
\urladdr{http://research.microsoft.com/$\sim$holroyd}

\date{24 September 2010}



\begin{document}

\begin{abstract}
The rotor walk on a graph is a deterministic analogue of random
walk.  Each vertex is equipped with a rotor, which routes the
walker to the neighbouring vertices in a fixed cyclic order on
successive visits. We consider rotor walk on an infinite rooted
tree, restarted from the root after each escape to infinity. We
prove that the limiting proportion of escapes to infinity
equals the escape probability for random walk, provided only
finitely many rotors send the walker initially towards the
root.  For i.i.d.\ random initial rotor directions on a regular
tree, the limiting proportion of escapes is either zero or the
random walk escape probability, and undergoes a discontinuous
phase transition between the two as the distribution is varied.
In the critical case there are no escapes, but the walker's
maximum distance from the root grows doubly exponentially with
the number of visits to the root. We also prove that there
exist trees of bounded degree for which the proportion of
escapes eventually exceeds the escape probability by
arbitrarily large $o(1)$ functions.  No larger discrepancy is
possible, while for regular trees the discrepancy is at most
logarithmic.
\end{abstract}

\maketitle

\section{Introduction}

The \df{rotor walk} is a derandomized variant of the random
walk on a graph $G$, defined as follows (see also
\cite{hlmppw,h-propp}, for example).  To each vertex of $G$ we
assign a fixed cyclic order of its neighbours, and at each
vertex there is a \df{rotor}, which points to some neighbour. A
\df{particle} is located at some vertex.  The particle location
and rotor positions evolve together in discrete time as
follows.  At each time step, the rotor at the particle's
current location is first incremented to point to the next
neighbour in the cyclic order, and then the particle moves to
this new neighbour.  The rotor walk is obtained by repeatedly
applying this rule.

In many settings, there is remarkable agreement between the
behaviour of rotor walk and expected behaviour of {\em random}
walk; see for example
\cite{cooper-doerr-spencer-tardos-0,cooper-spencer,
doerr-friedrich,h-propp,kleber,levine-peres-3,levine-peres,
levine-peres-2}.  However, the two processes can
also exhibit striking differences.  Landau and Levine
\cite{landau-levine} demonstrated this in the context of
recurrence and transience properties for regular trees (see
also \cite{cooper-doerr-friedrich-spencer}). In this article we
further investigate these issues for general trees.

Suppose that $G$ is an infinite graph with all degrees finite.
Let the particle start at a fixed vertex $\rt$, called the
\df{root}, and consider the rotor walk stopped at the first
return to $\rt$.  Either the particle returns to $\rt$ after a
finite number of steps, or it escapes to infinity without
returning to $\rt$, visiting each vertex only finitely many
times. In either case, the positions of the rotors after the
walk is complete are well defined.  We then start a new
particle at $\rt$ (without changing the rotor positions), and
repeat the above procedure, and so on indefinitely. Let $E_n$
be the number of particles that escape to infinity (as opposed
to returning to $\rt$) after $n$ rotor walks are run from $\rt$
in this way (so $0\leq E_n\leq n$).

Let $\trans=\trans(G)$ denote the probability that a simple
symmetric random walk on $G$ started at $\rt$ never returns to
$\rt$.  The following key result of Schramm, a proof of which is presented
in \cite[Theorem 10]{h-propp}, states that the rotor walk is in
a certain sense no more transient than the random walk.

\begin{thm}[Density bound; Oded Schramm]\label{oded}
For any graph with all degrees finite, any starting vertex
$\rt$, any cyclic orders of neighbours and any initial rotor
positions,
$$\limsup_{n\to\infty} \frac{\Esc_n}{n} \leq \trans.
$$
\end{thm}

The above result suggests the following questions.  When does
$E_n/n$ have a limit, and when is it equal to $\trans$?  When
is it strictly smaller? In the former case, how large can the
difference $E_n - \trans n$ be, as a function of $n$? The
answers in general depend both on the graph $G$ and the initial
positions of the rotors.  We will provide answers when the
graph is a tree.

Let $G=T$ be an infinite tree with vertex set $V$ and all
degrees finite. For notational convenience, we take $T$ to
denote the abstract tree {\em together with the cyclic orders}
of neighbours.  We will also always assume that the root $\rt$
has degree $1$ (if this condition is dropped, it is
straightforward to reformulate all our results appropriately).
For a vertex $v\neq \rt$, let $v^{(0)}$ be its \df{parent}
(i.e.\ its neighbour on the unique path to the root), and let
$v^{(1)},\ldots,v^{(b)}$ be its $b=b(v)$ \df{children} (i.e.\
its remaining neighbours), labeled so that the rotor at $v$
points to the neighbours in the cyclic order
$v^{(0)},v^{(1)},\ldots,v^{(b)}$.  (However, we will allow it
to start at an arbitrary point in this order.)  Write
$V_\rt:=V\setminus\{\rt\}$ and $\N=\{0,1,2\ldots\}$. A
\df{rotor configuration} is a map $r:V_\rt\to\N$, with $0\leq
r(v)\leq b(v)$ for all $v$, indicating that the rotor at vertex
$v$ points to $v^{(r(v))}$ (the rotor at $\rt$ always points to
its only child). Our principal object of study is the escape
number $E_n=E_n(T,r)$, defined to be the number of particles
that escape to infinity after $n$ successive rotor walks are
run from $\rt$ (as discussed above), with initial rotor
configuration $r$. See Figure \ref{example}.
\begin{figure}
\centering
\resizebox{3.6in}{!}{\includegraphics{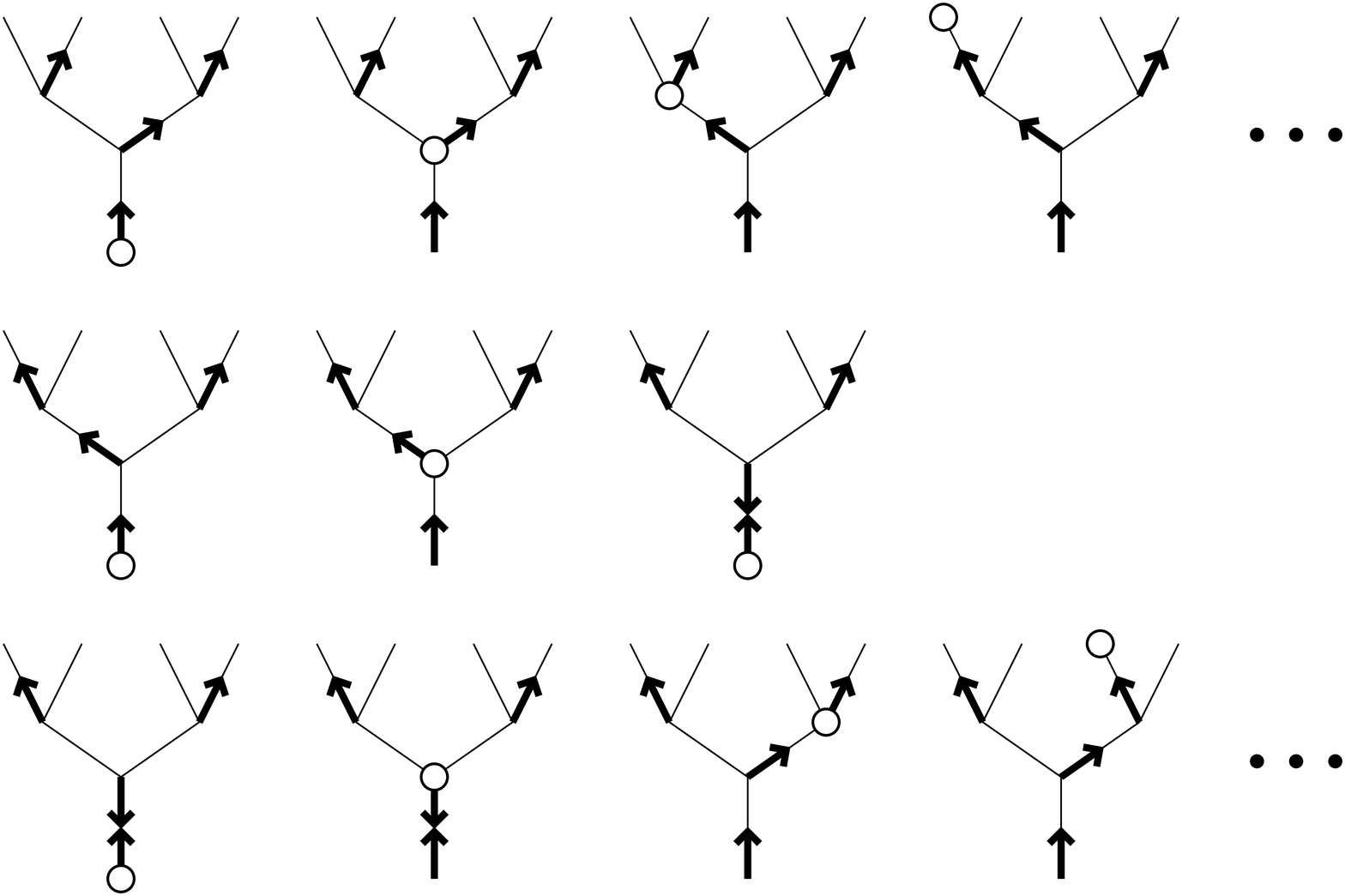}}
\caption{The rotor walk on an infinite binary tree (shown up to level 3).
The rotors (thick arrows) turn anticlockwise with positions $0$ (down),
$1$ (right), $2$ (left), and all non-root rotors have initial position $1$.
The first particle (white disc) escapes to infinity, the second returns
to the root, and the third escapes, so $E_3=2$.
} \label{example}
\end{figure}

Let $\down$ denote the rotor configuration in which each rotor points towards
the root, and let $\bb$ denote the rotor configuration in which each rotor
will {\em next} point towards the root, i.e.\ let $\down(v)=0$ and
$\bb(v)=b(v)$ for all $v\in V_\rt$.

\begin{thm}[Monotonicity, extremal configurations]\label{multi}
Let $T$ be a tree.
\begin{mylist}
\item If $r,s$ are rotor configurations satisfying $r(v)\leq
    s(v)$ for all $v\in V_\rt$, then
    $ \Esc_n(T,r)\geq \Esc_n(T,s)$ for all $n\geq 0.$
\item 
    We have $\Esc_n(T,\down)\geq \trans n$ for all $n\geq 0$.
\item 
    We have $\Esc_n(T,\bb)=0$ for all $n\geq 0$.
\end{mylist}
\end{thm}

\sloppypar It follows immediately from Theorem \ref{oded} and
Theorem \ref{multi}(ii) that $\lim_{n\to\infty} E_n(T,\down)/n
= \trans$. We establish the following much stronger result.

\begin{thm}[Limiting density]
\label{b-1} For any tree $T$, if $r$ is any rotor configuration having only
finitely many vertices $v$ with $r(v)=b(v)$, then
$$\lim_{n\to\infty} \frac{\Esc_n(T,r)}{n} = \trans.$$
\end{thm}

\sloppypar In the case of the binary tree $\TT_2$, Landau and
Levine \cite[Theorem~1.2]{landau-levine} characterized all
possible sequences $(E_n(\TT_2,r))_{n\geq 0}$ that arise as the
configuration $r$ varies.  Their result implies in particular
that for any $\alpha,\beta$ satisfying
$0\leq\alpha\leq\beta\leq\trans(=\tfrac12)$ there exists a
configuration $r$ such that $\liminf_{n\to\infty}
E_n(\TT_2,r)/{n} = \alpha$ and  $\limsup_{n\to\infty}
{E_n(\TT_2,r)}/{n} = \beta.$

Theorem \ref{oded} leaves open the possibility that the number
of escapes $E_n$ exceeds the expected number $\trans n$ for
random walk by an arbitrarily large $o(n)$ function (but no
more than this), and our next result shows that this is indeed
possible, even for a tree of bounded degree.
\begin{thm}[Large discrepancy]\label{T:many_escape}
  For any function $f$ that satisfies $f(n)=o(n)$ as $n\to\infty$,
  there exists a tree $T$ of maximum degree $3$ such that
  \[
  E_n(T,\down) \geq \trans n + f(n)
  \qquad \text{for all $n$ sufficiently large.}
  \]
\end{thm}
In fact we will prove Theorem \ref{T:many_escape} by exhibiting a recurrent
tree with the required property (so $\trans=0$).  It is easy to obtain an
example with any given $\trans\in(0,1)$ by adjoining a suitable transient
tree.

We next specialize to regular trees.  For $b\geq 2$, let
$\TT_b$ be the \df{$b$-ary} tree, in which every vertex $v\in
V_\rt$ has exactly $b$ children.  In contrast with
Theorem~\ref{T:many_escape}, in this case the number of escapes
differs from its expected value by at most a logarithmic term.

\begin{thm}[Regular trees]\label{log}
For the $b$-ary tree $\TT_b$, let $\mathbf{k}$ be the constant
configuration given by $\mathbf{k}(v)=k$ for all $v\in V_\rt$.
For $b\geq 2$ and $0\leq k< b$, the normalized discrepancy
$$\Delta_n:=\frac{E_n(\TT_b,\mathbf{k})-\trans n}{\log_b n}$$
satisfies
$$\liminf_{n\to\infty} \Delta_n=
\begin{cases}
0, & k=0;\\
- \frac{k-1}{b}, & k>0;
\end{cases}
\qquad
\limsup_{n\to\infty} \Delta_n=\frac{b-k-1}{b}.$$
\end{thm}

Recall that the remaining case $k=b$ is covered by Theorem
\ref{multi}(iii). The presence of the exceptional case $k=0$
for the liminf is perhaps surprising.  In particular, the
difference $\limsup\Delta_n-\liminf\Delta_n$ equals $(b-2)/b$
for $k>0$ but $(b-1)/b$ for $k=0$.  In the only case where this
difference is zero, namely $(b,k)=(2,1)$, we in fact have
$E_n=\lceil n/2\rceil$, as noted in \cite{landau-levine}.  The
proof of Theorem \ref{log} involves an explicit expression for
$E_n(\TT_b,\mathbf{k})$ for the general case, so further
refinements of the bounds are available.

Now consider a {\em random} initial configuration $R$ on the $b$-ary tree, in
which $(R(v))_{v\in V_\rt}$ are independent identically distributed random
variables having some distribution on $\{0,1,\ldots,b\}$.  Let $\P$ denote
probability and $\E$ expectation.  Assume that the distribution is not
deterministic (otherwise Theorems \ref{multi}(iii) and \ref{b-1} apply). Then
the behaviour of the rotor walk depends dramatically on the distribution.
\begin{thm}[Discontinuous phase transition]\label{phase}
For a non-deterministic i.i.d.\ rotor configuration $R$ on the
$b$-ary tree $\TT_b$, writing $v(\neq\rt)$ for an arbitrary
vertex, we have almost surely
\begin{mylist}
\item $\lim_{n\to\infty}
    \displaystyle\frac{E_n(\TT_b,R)}{n}=\trans \qquad\text{if
    }\E R(v) <b-1;$
    \vspace{2mm}
\item
$E_n(\TT_b,R)=0, \;\forall n\geq 0 \qquad\text{if }\E R(v) \geq b-1.$
\end{mylist}
\end{thm}

The dichotomy in Theorem \ref{phase} reflects the survival or
otherwise of a certain branching process.  In particular we
note that the archetypal example of i.i.d.\ uniformly random
rotors on the binary tree $\TT_2$ (i.e.\ $\P(R(v)=i)=1/3$ for
$i=0,1,2$) is critical ($\E R(v)=b-1$), so there are no
escapes.  This answers a question of Jim Propp (personal
communication). Note the contrast with the {\em deterministic}
case $R(v)\equiv b-1$, in which Theorem \ref{b-1} shows that
the proportion of escapes is $\trans$.

Finally, we show that in the critical case, particles typically
reach very great depths in the tree before returning to the
root.  This has implications for simulation: although all
particles return to the root, this will not be evident in any
finite simulation of practical size.

\begin{thm}[Maximum depth]\label{T:depth}
  Let $R$ be a non-deterministic i.i.d.\ configuration on the $b$-ary
  tree $\TT_b$, for $b\geq 2$. Let $D_n$ be the maximum distance from
  the root reached by the first $n$ particles.  Then for some constants
  $c,C\in(0,\infty)$ depending on the distribution of $R$,
\begin{mylist}
\item
$\P\bigl[ e^{e^{cn}} < D_n < e^{e^{Cn}} \bigr]\to 1 \text{ as }
  n\to\infty \quad\text{if }\E R(v)= b-1; $
  \vspace{1mm}
\item $\P\bigl[n \leq D_n < Cn \bigr]\to 1 \text{ as }
  n\to\infty \quad\text{if }\E R(v)> b-1.$
\end{mylist}
\end{thm}

\subsection*{Further Remarks}

The rotor walk was first introduced in \cite{pddk}, under the
name ``Eulerian walkers model''.  For a general account of the
model and its history, see \cite{hlmppw}.  In \cite{h-propp} it
was shown that on finite graphs, and some infinite graphs, many
quantities associated with random walk (hitting probabilities,
hitting times, stationary measures, etc.) are very well
approximated by rotor walk counterparts.  In
\cite{cooper-doerr-spencer-tardos,cooper-spencer} a different
variant is studied, in which several rotor walks are run
simultaneously for a fixed number of steps; again it is shown
that this model closely approximates the expected behaviour of
random walk.  The rotor aggregation model (introduced by Propp)
is a rotor version of internal diffusion-limited aggregation,
in which successive rotor walks are started at a fixed vertex
and run until they first reach a previously unoccupied vertex.
On the integer lattice, the resulting occupied set rapidly
approaches a spherical ball, as proved in
\cite{levine-peres,levine-peres-2}.

Rotor walks on trees have been studied in detail by Landau and
Levine \cite{landau-levine} (as well as earlier in
\cite{cooper-doerr-friedrich-spencer,dtw}). In particular, on
the binary tree $\TT_2$ they identified the escape sequences
for the constant configurations $\down$ (no escapes) and
$\mathbf{1}$ (alternating escapes and returns), and furthermore
provided an exact characterization for the set of possible
escape sequences $(E_n(\TT_2,r))_{n\geq 0}$.  This
characterization is more conveniently stated in terms of the
binary escape sequence $e=(e_1,e_2,\ldots)$ where $e_n$ is the
indicator that particle $n$ escapes (so $E_n=\sum_{i=1}^n
e_i$).  The result of \cite{landau-levine} is that $e$ is an
escape sequence for some rotor configuration on $\TT_2$ if and
only if for each $k\geq 2$, each interval of $e$ of length
$2^k-1$ contains at most $2^{k-1}$ ones.  In particular it
follows that if $e$ is a possible escape sequence, then so is
any $e'$ satisfying $e'_n\leq e_n$ for all $n$ (this implies
that $E_n/n$ may have arbitrary liminf and limsup in $[0,1/2]$,
as remarked earlier).  We do not know whether this holds for
general trees.

The rotor aggregation model on regular trees was also studied
in \cite{landau-levine}, where it was shown that the occupied
set is exactly a ball for acyclic initial rotor configurations.
The proof makes use of the sandpile group (see also
\cite{hlmppw,levine-group}).  We remark that the rotor
aggregation model on a tree $T$ can be encoded as an ordinary
rotor walk process on a different tree $T'$.  We construct $T'$
by attaching a singly infinite path to each vertex of $T$, with
all the rotors on the new vertices pointing towards $T$, and
each rotor at an original vertex of $T$ set to point next along
the added path, and then in the direction it would have pointed
next in $T$.  When a particle reaches a vertex of $T$ for the
first time, it immediately escapes along the path (signifying
an aggregation).  On subsequent visits to the vertex, the
particle makes only finite excursions along the path, and then
continues as the rotor walk on $T$ would have done.

We briefly mention two somewhat degenerate classes of trees
which we will not focus on.  First, suppose some vertices have
no children (this is not forbidden by our assumptions). Then we
may remove every vertex that has only finitely many descendants
without affecting the behaviour of the rotor walk or the random
walk, so there is no loss generality in discarding this case.
Second, suppose a vertex $w$ has infinitely many children
(which is forbidden by our assumptions).  Under the mild
condition that the rotor at $w$ will point to infinitely many
children whose subtrees have the property that the first
particle sent there escapes, then $w$ acts as a {\em sink},
which is to say that every particle sent there escapes (and no
other information about $w$'s descendants is relevant). Finite
trees with sinks are considered in Section \ref{s:trunc}, and
many of our results could be extended without difficulty to
infinite trees with sinks, but we have chosen not to pursue
this.

Some applications of rotor walks may be found in
\cite{doerr-friedrich-sauerwald,dtw,friedrich-gairing-sauerwald},
and further recent progress is reported in
\cite{friedrich-levine,kager-levine,levine-peres-3}.

A key tool for our proofs will be a recursive formula for the
escape numbers $E_n(T,r)$, called the {\em explosion formula},
which generalizes ideas in \cite{landau-levine} and is
introduced in Section~\ref{s:explosion}.  This will be used to
prove monotonicity (Theorem~\ref{multi}(i)), the bounds for
regular trees (Theorem~\ref{log}), and a number of other
results.  The large-discrepancy tree in Theorem
\ref{T:many_escape} will constructed from {\em brushes}
(higher-dimensional analogues of combs), which we will analyze
in Section~\ref{s:brush}. Theorem \ref{b-1} will use a
martingale argument for random walks on trees (Section
\ref{s:b-1}).  As remarked earlier, the discontinuous phase
transition in Theorem~\ref{phase} is related to a branching
process.   A simple argument will show that there are no
escapes in the subcritical and critical cases (ii) (see
Section~\ref{s:live}); for the supercritical case (i), the
proof makes use of the Abelian property
(Section~\ref{s:random}). The doubly-exponential depth in the
critical case, Theorem~\ref{T:depth}(i), reflects the fact that
the set of visited vertices is (a minor variation of) a
branching process whose offspring distribution is itself the
total population size of a critical branching process; see
Section~\ref{s:depth}.

We will prove several other results in addition to those
mentioned above, including: an expression for the total number
of escapes (Proposition \ref{live-ends}); Lipschitz,
monotonicity, and invariance conditions for the escape numbers
in terms of rotors, trees, and rotor orders
(Propositions~\ref{lip}, \ref{subgraph}, and \ref{rot-mech}); a
self-majorization property (Proposition \ref{self-maj}); and a
surprising regularity result for the behaviour of rotor walks
on finite regular trees (Proposition \ref{exact}) with an
arbitrary rotor configuration.

\section{Live paths}
\label{s:live}

In this section we prove Theorem~\ref{multi}(iii) and
Theorem~\ref{phase}(ii). We introduce the concept of live
paths, which will also be important for the later proof of
Theorem \ref{phase}(i).

\begin{proof}[Proof of Theorem~\ref{multi}(iii)]
  With initial configuration $\mathbf{b}$, suppose for a contradiction that
  some particle escapes to infinity. Let particle $n$ be the first such
  particle. Clearly there exists a vertex $v$ such that particle $n$ visits
  both $v$ and one of its children, and $v$ was not visited by any previous
  particle. Since on the first visit to $v$ the particle is sent back
  towards the root, particle $n$ must move from the parent $v^{(0)}$ to $v$
  at least twice.

  Let $u$ be the first vertex to receive particle $n$ twice from its
  parent. The parent cannot be $\rt$, since by assumption particle $n$
  never returns to $\rt$. Hence the rotor at the parent $u^{(0)}$ must have
  made a full rotation between the two traversals of particle $n$ from
  $u^{(0)}$ to $u$. In particular, $u^{(0)}$ sent the particle to its
  parent in between, and thus $u^{(0)}$ received the particle twice from
  its parent before $u$ did, a contradiction.
\end{proof}

Given a tree $T$ and a rotor configuration $r$, a \df{live
path} is an infinite sequence of vertices $\rt\neq
x_1,x_2,\ldots$, each the parent of the next, such that for
each $i$, the $k$ such that $x_{i+1}=(x_i)^{(k)}$ satisfies
$r(x_i)<k$ -- in other words, $x_i$ will send the particle
forward to $x_{i+1}$ before sending it back towards the root.
An \df{end} of $T$ is an infinite sequence of vertices
$\rt=x_0,x_1,\ldots$, each the parent of the next. Call an end
\df{live} if the subsequence $(x_i)_{i\geq j}$ starting at one
of its vertices is a live path. Let
$E_\infty(T,r):=\lim_{n\to\infty} E_n(T,r)$ be the total number
escapes ever (which may be a non-negative integer or $\infty$).

\begin{prop}[Live ends]\label{live-ends}
The total number of escapes $E_\infty(T,r)$ equals the number of live ends in
the initial rotor configuration $r$.
\end{prop}

\begin{proof}
  Suppose that $\rt\neq x_1,x_2,\ldots $ is a live path, and that the
  particle is currently located at $x_1$. We claim that the particle will
  escape to infinity without ever visiting the parent of $x_1$. Suppose on
  the contrary that it visits the parent of $x_1$ at some (finite) time. By
  the definition of a live path, it must previously have visited $x_2$.
  Similarly, considering the rotor at $x_2$, we deduce that it must have
  visited $x_3$ before moving from $x_2$ to $x_1$. Repeating this argument
  shows that the particle must have visited the infinite set of vertices
  $x_2,x_3,\ldots$ before the parent of $x_1$, a contradiction which proves
  the claim.

Now suppose that the set of live ends is $L$, and a particle is started at
$\rt$.  If the particle returns to $\rt$ then the set of live ends is still
$L$ at this point, because only finitely many vertices have been visited.
Suppose on the other hand that the particle escapes without returning to
$\rt$.  The set of all vertices visited by this particle consists of exactly
one end $\eta=(x_0,x_1,\ldots)$, together with finite trees rooted at
vertices of $\eta$.  We claim that the particle never backtracks on $\eta$,
i.e.\ after its first visit to $x_{i+1}$, it never visits $x_i$.  This holds
for $i=1$ by assumption.  Suppose that it fails for some $i>1$, i.e.\ the
particle visits $x_i$, then $x_{i+1}$, then $x_i$ again.  Before another
visit to $x_{i+1}$ (which must eventually occur), the rotor at $x_i$ must
make a full rotation, so the particle must visit $x_{i-1}$, implying that the
claim fails for $i-1$ also. Thus, the non-backtracking claim is proved.  It
follows that $(x_1,x_2\ldots)$ was originally a live path, and in particular
$\eta\in L$. Just after the particle has escaped, all rotors on $\eta$ point
along $\eta$ away from the root, so $\eta$ is no longer live; however, any
other end has been visited at only finitely many of its vertices;
consequently the set of live ends is now $L\setminus\{\eta\}$.

Finally observe that, in the infinite sequence of restarts at
$\rt$ implicit in the definition of $E_\infty$, each vertex is
visited infinitely often (this holds for the child of $\rt$,
therefore for each of its children, etc.).

The result now follows from the above claims.  So long as there are live
ends, there will eventually be a visit to a live path and an escape, which
results in the set of live ends being depleted by one; at all other times
this set remains unchanged.
\end{proof}

As usual, if vertex $u$ is on the unique self-avoiding path from $v$ to
$\rt$, then we say that $u$ is an \df{ancestor} of $v$, and $v$ is an
\df{descendant} of $u$.

\begin{proof}[Proof of Theorem \ref{phase}(ii)]
Suppose that $R$ is an i.i.d.\ random configuration.  Let $u\neq \rt$ be a
vertex and let $v=u^{(k)}$ be one of its children. Call the child $v$
\df{good} if $R(u)<k$.  Thus, $x_1,x_2,\ldots$ is a live path if and only if
$x_2,x_3,\ldots$ are all good.  On the other hand, the number of good
children of $u$ is precisely $b-R(u)$.  Hence, for any fixed $u\neq\rt$, the
set of descendants $w$ of $u$ that can be reached from $u$ via a path of good
vertices (including $w$ but not necessarily $u$) forms a Galton-Watson
branching process with offspring distribution that of $b-R(u)$.  If $R$ is
non-deterministic and $\E R(v) \geq b-1$ then this process dies out, hence
a.s.\ there are no live paths, and therefore no escapes by
Proposition~\ref{live-ends}.
\end{proof}

\section{Truncation}
\label{s:trunc}

It will sometimes be useful to approximate infinite trees with finite ones.
Let $T$ be a {\em finite} tree with a root $\rt$ of degree $1$, and let
$S\not\ni\rt$ be a non-empty subset of the leaves; we call elements of $S$
\df{sinks}. Consider a rotor walk with initial configuration $r$, started at
$\rt$ and stopped at the first entry to $S\cup\{\rt\}$.  Let $E_n(T,S,r)$ be
the number of particles that stop at $S$ when $n$ such walks are run in
succession, starting with configuration $r$.

The \df{level} of a vertex is its graph-distance from the root.
Given an infinite tree $T$ and a positive integer $h$, define
the truncated tree $T^h$ to be the subgraph induced by the set
of vertices at levels at most $h$, and let $S^h$ be the set of
vertices at level exactly $h$.  Also let $r^h$ denote the rotor
configuration $r$ restricted to the vertices of $T^h$
(excluding $S^h$ and $\rt$).

\begin{lemma}[Truncation]
\label{trunc}
For any tree $T$, any rotor configuration $r$, and any $n$,
$$E_n(T^h,S^h,r^h)\searrow E_n(T,r) \quad\text{as }h\to\infty,$$
(i.e.\ the left side is decreasing in $h$, and equals the right side for $h$
sufficiently large).
\end{lemma}

Lemma~\ref{trunc} is proved in \cite[Lemmas 18,19]{h-propp};
indeed it holds for arbitrary graphs, and in the more general
setting of rotor walks associated with Markov chains.  For the
reader's convenience we summarize the argument here.  The
convergence follows from the stronger statement that the number
of visits to any given vertex converges similarly, which is
proved by induction on $n$.  Monotonicity is a consequence of
the Abelian property for rotor walks on a finite graph with a
sink (see e.g.\ \cite{hlmppw}).

\section{Rotors pointing towards the root}
\label{s:down}

In this section we will prove Theorem \ref{multi}(ii).  We will prove a
version for finite trees and then appeal to Lemma~\ref{trunc}.

Consider a finite tree $T$ in which the root $\rt$ has exactly one child,
and let $S\not\ni\rt$ be a non-empty subset of the leaves of $T$. Construct
a directed graph $G=G(T,S)$ as follows. Replace each edge of $T$ not
incident to $S$ with two directed edges, one in each direction. Replace
each edge incident to $S$ with a single directed edge towards $S$. Add
directed edges from each vertex in $S$ to $\rt$. Note that if a particle
performs some walk on $T$ and is returned to $\rt$ whenever it reaches $S$,
then its trajectory is a (directed) path in $G$.

Consider any finite directed path $\pi$ in the graph $G$, starting and ending
at the same vertex.  Let $v\neq\rt$ be any vertex of the original tree $T$,
and let $u=v^{(0)}$ be its parent.  Let $n_v$ be the number of times the path
traverses the directed edge from $u$ to $v$, and $m_v$ the number of
traversals from $v$ to $u$.  Let $p_v$ be the probability that the simple
random walk on $T$ started at $v$ hits $u$ before $S$.  Define the
discrepancy
\[ \delta_v := p_v n_v -m_v. \]

\begin{lemma}\label{delta}
Fix a finite tree and a path $\pi$ as above, and let $v\not\in S\cup\{\rt\}$
be a vertex, with children $v^{(1)},\ldots,v^{(b)}$. Abbreviate $\delta_v$ to
$\delta$, and $\delta_{v^{(i)}}$ to $\delta_i$, and similarly for $n$, $m$,
and $p$.  We have
$$\delta=p \sum_{i=1}^b\big[\delta_i+(1-p_i)(n_i-m)\big].$$
\end{lemma}

\begin{proof}
First consider the random walk started at $v$. At each visit to $v$, the walk
either returns immediately to the parent with probability $1/(b+1)$, or
escapes to $S$ via $v^{(i)}$ with probability $(1-p_i)/(b+1)$, or otherwise
returns to $v$ without doing either.  Therefore the probability $p$ of
returning to the parent before escaping to $S$ is given by
\begin{equation}\label{p0}
p=\frac{1/(b+1)}{1/(b+1)+\sum_{i=1}^b (1-p_i)/(b+1)}=
\frac{1}{1+\sum_{i=1}^b (1-p_i)}.
\end{equation}

Turning to the path $\pi$, since the numbers of arrivals and departures at
$v$ are equal we have $n+\sum_{i=1}^b m_i= m+\sum_{i=1}^b n_i$, or
equivalently,
\begin{equation}\label{n-m}
\sum_{i=1}^b (n_i-m_i)=n-m.
\end{equation}

Now using the definition of $\delta_i$, followed by
\eqref{n-m}, \eqref{p0}, and the definition of $\delta$, we
obtain
\begin{multline*}
\sum_{i=1}^b \big[\delta_i+(1-p_i)(n_i-m)\big]
=\sum_{i=1}^b \big[ (n_i-m_i)-m(1-p_i)\big] \\
=\ n-m-m\sum_{i=1}^b(1-p_i) =n-m/p  =\delta/p,
\end{multline*}
as required.
\end{proof}

\begin{proof}[Proof of Theorem \ref{multi}(ii)]
Apply Lemma~\ref{delta} to a rotor walk on a finite tree $T$
with initial configuration $\down$, started at $\rt$, returned
to $\rt$ after every visit to $S$, and run up until some visit
to $\rt$.  In this case, considering the rotor at $v$ shows
that $n_i\geq m$, therefore the lemma gives $\delta\geq
\sum_{i=1}^b \delta_i$. For any $w\in S$ we have $p_w=m_w=0$,
therefore $\delta_w=0$.  Hence by applying the previous
inequality iteratively we deduce that $\delta_\nx\geq 0$, where
$\nx$ is the child of the root.  In other words, for a finite
tree, $E_n(T,S,\down)\geq \trans(T,S)\, n$, where
$\trans(T,S)=1-p_\nx$ is the probability that a random walk
started at $\rt$ hits $S$ before returning to $\rt$.

Now for an infinite tree $T$ we apply Lemma~\ref{trunc}.  By the above we
have for all $n$ and $h$,
$$E_n(T^h,S^h,\down)\geq \trans(T^h,S^h)\, n,$$
while
$$\trans(T^h,S^h)\searrow\trans(T) \quad\text{as }h\to\infty.$$
Hence the lemma gives $E_n(T,\down)\geq \trans(T) \, n$ as required.
\end{proof}

\section{Explosion formula}
\label{s:explosion}

We introduce a recursive formula for the escape numbers
$E_n(T,r)$, from which we will deduce Theorem \ref{multi}(i)
and many other results. It will be convenient to work with the
indicator $e_n=e_n(T,r):=\ind[\text{particle $n$ escapes to
infinity}]$, so that $E_n=\sum_{i=1}^n e_i$.  Let $e(T,r)$
denote the binary \df{escape sequence} $(e_1,e_2,\ldots)$.  We
sometimes abbreviate a sequence $(a_1,a_2,\ldots)$ to $a_1
a_2\cdots$.

For a tree $T$, let $\nx$ be the unique child of the root
$\rt$, and call $\nx$ the \df{base} of the tree.  For any
vertex $v\neq \rt$ of $T$, the \df{subtree} $T_v$ \df{based at
$v$} is defined to be the subgraph of $T$ induced by the parent
$v^{(0)}$ of $v$ together with $v$ and all its descendants. The
subtree $T_v$ has root $v^{(0)}$ and base $v$.  For a rotor
configuration $r$ on $T$, we write $r_v$ for its restriction to
$T_v$ (more precisely, to the non-root vertices of $T_v$).  If
$\nx$ has children $\nx^{(1)},\ldots,\nx^{(b)}$, we abbreviate
$T_{\nx^{(i)}}$ to $T_i$ and $r_{\nx^{(i)}}$ to $r_i$, and call
$T_1,\ldots,T_b$ the \df{principal branches} of $T$.

Similarly if $T$ is a finite tree with a set of sinks $S$, as in Section
\ref{s:trunc}, we write $S_v:=S\cap V(T_v)$ , and $S_i:=S_{\nx^{(i)}}$.  We
also write $e_n(T,S,r):=\ind[\text{particle $n$ escapes to $S$}]$, so again
$E_n=\sum_{i=1}^n e_i$.

Let $\N_+=\{1,2\ldots\}.$  For sequences in ${\N}^{\N_+}$ we denote addition
by $(a_1,a_2,\ldots)+(b_1,b_2,\ldots):=(a_1+b_1,a_2+b_2,\ldots)$, and define
the \df{shift} operator $\theta$ by
$$\theta(a_1,a_2,\ldots):=(0,a_1,a_2,\ldots)$$
and the \df{explosion} operator $\ex$ by
$$\ex(a_1,a_2,\ldots):=(1^{a_1},0,1^{a_2},0,\ldots)$$
where $1^k$ denotes a string of $k$ $1$s.  We define the \df{majorization}
order $\preceq$ on sequences by $(a_1,a_2,\ldots)\preceq(b_1,b_2,\ldots)$ if
and only if $\sum_1^n a_i\leq \sum_1^n b_i$ for all $n$.

\begin{thm}[Explosion formula]\label{explosion}\
\begin{mylist}
\item Let $T$ be a tree with base $\nx$ and principal branches
    $T_1,\ldots,T_b$, and fix a rotor configuration $r$.  Then
\begin{equation}\label{it}
e(T,r)=\ex\bigg(\sum_{i=1}^{r(\nx)} \theta e(T_i,r_i)+ \sum_{i=r(\nx)+1}^b
e(T_i,r_i)\bigg).
\end{equation}
\item If $T$ is a finite tree with a set of sinks $S$,
    \eqref{it} holds similarly for $e(T,S,r)$ and
    $e(T_i,S_i,r_i)$.
\item In the case of an infinite tree as in (i), the collection
    of escape sequences
    $(e(T_v,r_v))_{v\in V_\rt}$ is the unique maximal
    $\{0,1\}^{\N_+}$-valued solution to the system of equations
    obtained by applying \eqref{it} to each of the subtrees
$(T_v)_{v\in V_\rt}$ (in place of $T$).  Here maximality means that if
$(e'(v))_{v\in V_\rt}$ is another solution, then $e(T_v,r_v)\succeq e'(v)$
for all $v$.
\end{mylist}
\end{thm}

We remark that the system of equations in (iii) above does not
generally have a unique solution; for example $e(T_v,r_v)\equiv
(0,0,\ldots)$ is always a solution. It is not {\em a priori}
obvious that the system of equations has a maximal solution in
the sense asserted in (iii) (as opposed to the weaker sense of
being majorized by no other solution).

We will use Theorem \ref{explosion} to prove Theorem
\ref{multi}(ii) and Theorem~\ref{log}, as well as the following
propositions, of which the first two will be used later, and
all are of independent interest.

\begin{prop}[Finite modification]\label{lip}
If $r$ and $r'$ are rotor configurations differing at only finitely many
vertices, then for all $n$,
$$|E_n(T,r)-E_n(T,r')|\leq \sum_{v\in V_\rt} |r(v)-r'(v)| .$$
\end{prop}

\begin{prop}[Subtrees]\label{subgraph}
  If the tree $T'$ is a subgraph of the tree $T$, with the same root, then
  $e(T',\down) \preceq e(T,\down)$ (i.e.\ for all $n$ we have
  $E_n(T',\down)\leq E_n(T,\down)$).
\end{prop}

\begin{prop}[Rotor orders]\label{rot-mech}
If $T,T'$ have the same underlying graph, but different cyclic
orders of neighbours, and $r$ is a rotor configuration
satisfying $r(v)\in\{0,b(v)\}$ for each $v$ (such as $\down$),
then $E_n(T,r) = E_n(T',r)$ for all $n$.
\end{prop}

\begin{prop}[Non-uniqueness]\label{non-uniq}
There exist uncountably many distinct trees with the same
escape sequence for initial configuration $\down$.
\end{prop}

For sequences $a,b$ we write $a\prec b$ if $a\preceq b$ but
$a\neq b$.

\begin{prop}[Self-majorization]\label{self-maj}
Let $T$ be any infinite tree and let $e=e(T,\down)$.  For any
$n>0$ we have the strict majorization $e\succ
(e_{n+1},e_{n+2},\ldots)$.
\end{prop}

The above proposition implies in particular that $e(T,\down)$
cannot be a periodic sequence.  Periodicity is possible for
other rotor configurations (for example the constant
configuration $\mathbf{1}$ on the binary tree has
$e=101010\cdots$, as noted in \cite{landau-levine}). Not every
sequence $e$ satisfying the self-majorization condition of
Proposition~\ref{self-maj} is equal to $e(T,\down)$ for some
tree $T$. For example, Theorem \ref{explosion} and Proposition
\ref{self-maj} may be used to check that no sequence starting
$110010101010\cdots (=\ex(201111\cdots ))$ can arise in this
way. It is an open problem to characterize the set of all
possible sequences $e(T,\down)$ as $T$ varies.

Now we turn to the proofs.

\begin{proof}[Proof of Theorem \ref{explosion}]
We suppress the rotor configuration in the notation $e(T_v,r_v)$ for escape
sequences, since it will always be the appropriate restriction.

Part (i) is an elementary consequence of the behaviour of the rotor at
$\nx$. Suppose first that $r(\nx)=0$, so that $\nx$ will send the particle
to each of its children before returning it to $\rt$. When it is sent to
the $i$th child, either the particle escapes (and then is returned to $\nx$
via $\rt$), or it is returned to $\nx$ by the child, according to the value
of $e_1(T_i)$. In either case, the rotor at $\nx$ is then incremented and
we proceed to the next child. Finally, the rotor points to $\rt$, and we
get a return to $\rt$. Thus $e(T)$ starts with the sequence $1^k 0$, where
$k=\sum_{i=1}^b e_1(T_i)$. At this point the rotor is in its initial
position, and each principal branch has received a particle once from
$\nx$, so the process repeats using $e_2(T_i)$, and so on, and we deduce
$e(T)=\ex \sum_{i=1}^{b} e(T_i)$ as required in this case. The case of
general $r(\nx)$ is similar, except that we must first apply the shift
$\theta$ to the sequences $e(T_i)$ of those branches to which the rotor
will not point before pointing to $\rt$, because they do not contribute any
escapes to the first block of 1s of $e(T)$.

An identical argument to the above gives the finite case (ii). To prove
(iii), consider the truncated system $(T^h,S^h,r^h)$ (as defined in Section
\ref{s:trunc}). Applying (ii) to every subtree yields a system of equations
giving the escape sequence of each subtree in terms of those of its principal
branches.  Together with the boundary condition that $e(T^h_s)=111\cdots$ for
every $s\in S^h$, these equations determine all the escape sequences.  Note
that $\ex$ is an increasing function with respect to the order $\preceq$ (on
both the domain and the range).  Therefore if $e'$ is another
$\{0,1\}^{\N_+}$-valued solution to the system of equations, proceeding
iteratively starting from the sinks, we deduce that $e'(v)\preceq e(T^h_v)$
for each subtree. Finally we apply Lemma~\ref{trunc} to deduce that the same
conclusion holds for the infinite tree.
\end{proof}

\begin{proof}[Proof of Theorem \ref{multi}(i)]
For a finite tree with sinks, the required monotonicity follows from the
observation that the right side of \eqref{it} is decreasing in $r(\nx)$ and
increasing in $e(T_i)$, where sequences are ordered by majorization.  To
deduce the infinite case, apply Lemma~\ref{trunc}.
\end{proof}

\begin{proof}[Proof of Proposition \ref{lip}]
The main observation, which is straightforward to check, is
that if $a,a'$ is a pair of integer sequences satisfying
$\sum_{i=1}^n (a_i-a'_i) \in \{0,1\}$ for all $n$, then the
pair $\ex(a),\ex(a')$ satisfies the same condition.

Suppose first that $r$ and $r'$ differ only at a single vertex
$v$, where $r'(v)=r(v)+1$. In this case the explosion formula
\eqref{it} applied to $T_v$ gives $e(T_v,r_v)=\ex(a+e)$ and
$e(T_v,r'_v)=\ex(a+\theta e)$ for a non-negative sequence $a$
and a binary sequence $e$.  Hence by the observation above,
$E_n(T_v,r_v) - E_n(T_v,r'_v) \in \{0,1\}$ for all $n$.  If the
last statement holds for some vertex $u$ (in place of $v$),
then by the above observation again, it also holds for its
parent $u^{(0)}$.  Applying this iteratively shows that it
holds for $\nx$.

  The required inequality now follows by repeatedly applying the
  argument above, since we may move from $r$ to $r'$ via a sequence of
  configurations, in which one rotor is incremented or decremented at each
  step.
\end{proof}

\begin{proof}[Proof of Proposition \ref{subgraph}]
Let $r$ be a rotor configuration on $T$ given by
$$r(v)=\begin{cases}
0, &v\in V(T');\\
b(v), &v\notin V(T').
\end{cases}$$
In the rotor walk on $T$ with initial configuration $r$, whenever the
particle enters a subtree that is not present in $T'$, it is eventually
returns to the parent without escaping, by Theorem \ref{multi}(iii).
Therefore $e(T,r)=e(T',\down)$.  The result now follows by Theorem
\ref{multi}(i).
\end{proof}

\begin{proof}[Proof of Proposition \ref{rot-mech}]
If $r(\nx)\in\{0,b(\nx)\}$, the right side of \eqref{it} is unchanged by
reordering the principal branches.  Iterate to get the result for finite
trees, and apply Lemma~\ref{trunc} for the infinite case.
\end{proof}

\begin{proof}[Proof of Proposition \ref{non-uniq}]
  Let $a=(a_1,a_2,\ldots)$ be a strictly increasing sequence of
  non-negative integers, and let $T(a)$ be a `thinned comb' --
  a tree consisting of a
  singly-infinite path $(\rt,x_0,x_1,\ldots)$, together with additional
  singly-infinite paths attached at each of the vertices
  $\{x_{a_i}\}_{i\geq 1}$. It is straightforward to
  check that
$$e(T(a),\down)=1\,0^{a_1}\,1\,0^{a_2}\,1\,\cdots.$$
Thus there exist uncountably many pairs of such sequences
$(a,b)$ with the property that
$$e(T(a),\down)+e(T(b),\down)
= 2 \; 1 \; 0^1 \; 1 \; 0^3 \; 1 \; 0^7 \; 1 \; \cdots\
0^{2^i-1} \; 1 \cdots $$
since we may choose arbitrarily which
1s come from $a$ and which from $b$. For any such pair we may
form a tree by joining $T(a)$ and $T(b)$ together at their
roots, and adding a new edge from this vertex to a new vertex,
which we designate the new root.  The explosion formula shows
that the resulting trees all give the same escape sequence.
\end{proof}

\begin{proof}[Proof of Proposition \ref{self-maj}]
As remarked in the introduction, we may assume without loss of
generality that every vertex has at least one child, i.e.\
every subtree is infinite.

For the weak inequality $e\succeq (e_{n+1},e_{n+2},\ldots)$,
note that the rotor configuration $r$ after $n$ particles have
been added obviously satisfies $r(v)\geq \down(v)$ for all
$v\in V_\rt$, and apply Theorem \ref{multi}(i).

It remains to prove $e\neq(e_{n+1},e_{n+2},\ldots)$. This
indeed holds for $n=1$, because equality would imply that $e$
is a constant sequence, which is impossible since the first
particle escapes but some do not. For a general $n>0$, consider
the configuration $r$ after $n$ rotor particles have been run,
and let $v$ be some vertex that has been visited exactly once
at that time (such a vertex exists by the proof of
Proposition~\ref{live-ends}).  From the $n=1$ case just
considered, the subtree $T_v$ based at $v$ has the property
that its subsequent escape sequence is strictly majorized by
the initial sequence, $e(T_v,r_v)\prec e(T_v,\down_v)$.  On the
other hand, for {\em every} vertex $u$ we have the weak
inequality $e(T_u,r_u)\preceq e(T_u,\down_u)$. The right side
of the explosion formula \eqref{it} is strictly monotone in the
sense that: if for two rotor configurations $r,s$ we have
$r(\nx)\geq s(\nx)$, while for each principal branch we have
$e(T_i,r_i)\preceq e(T_i,s_i)$, with one of these majorizations
being strict, then $e(T,r)\prec e(T,s)$. Applying this
iteratively to the sequence of ancestors of $v$ gives the
result.
\end{proof}

\section{Constant configurations on regular trees}
\label{s:regular}

Theorem \ref{log} on regular trees is a consequence of the
following result. Fix $b\geq 2$, and for a positive integer
$t$, write $z_t$ for the number of trailing zeros in its
base-$b$ expansion and $\ell_t$ for the last nonzero digit, so
$t\equiv \ell_t b^{z_t} \pmod {b^{z_t+1}}$ with
$\ell_t\in\{1,\dots,b-1\}$.

\begin{prop}[Constant rotor configurations]\label{regular}
Let $\mathbf{k}$ be the constant rotor configuration on the
$b$-ary tree $\TT_b$ given by $\mathbf{k}(v)=k$ for all $v$.
The escape sequence is given as follows.
$$e(\TT_b,\mathbf{k})=\begin{cases}
  1 \;  0^{z_1} \; 1 \; 0^{z_2} \; 1 \; 0^{z_3} \cdots & \text{if } k=0;\\
  1 \;  0^{\ind[\ell_1=b-k]} \; 1 \; 0^{\ind[\ell_2=b-k]} \; 1 \; 0^{\ind[\ell_3=b-k]}
  \cdots &\text{if } 0<k<b.
\end{cases}
$$
\end{prop}
Recall that $e(\TT_b,\mathbf{b})=000\cdots$ by Theorem \ref{multi}(iii). It
is interesting that $e(\TT_b,\mathbf{k})$ contains consecutive zeros when
$k=0$, but not when $0<k<b$.

\begin{proof}[Proof of Proposition \ref{regular}]
By Theorem \ref{explosion}(i), the escape sequence $e=e(\TT_b,\mathbf{k})$
satisfies
\begin{equation}
e = \ex\bigl(k\,\theta e +(b-k)\,e \bigr). \label{ex-reg}
\end{equation}
Also, for $k<b$ it is clear that the first particle escapes,
i.e.\ $e_1=1$. It turns out that these facts are sufficient to
determine $e$ (we do not need to make explicit use of the
maximality in Theorem \ref{explosion}(iii)).

For the case $k=0$, we note that if $e$ is any sequence starting with
$$1 \; 0^{a_1} \; 1 \; 0^{a_2} \; 1 \; \cdots \; 0^{a_m} \; 1,$$
then $\ex\bigl(b\,e \bigr)$ starts with
\[
1^b \; 0^{a_1+1} \; 1^b \; 0^{a_2+1} \; 1^b \; \cdots \; 0^{a_m+1} \; 1^b.
\]
Therefore, we can obtain the sequence $e(\TT_b,\down)$ by starting with the
singleton sequence $1$ and repeatedly applying the above transformation.
Each finite sequence that results is necessarily a prefix of the next, and of
the desired infinite sequence.  Since $1^r = 1 \, 0^0\, 1 \,0^0 \cdots 1$, it
is easy to deduce that the claimed expression for $e(\TT_b,\down)$ holds.

The case $0<k<b$ is similar but a little more complicated.  Write $X:=1\,0$,
and suppose $e$ is any sequence starting with a concatenation of the form
$$ Y_1 \, Y_2 \, \cdots \, Y_m, $$
where each $Y_i$ is either $X$ or $1$.  Then it is straightforward to check
that substituting $e$ into the right side of \eqref{ex-reg} gives a sequence
starting
$$ W \, Y_1 \, W \, Y_2 \, \cdots \, W \, Y_m,$$
where $W:=1^{b-k-1} \,  X \,  1^{k-1}$.  Now the claimed
expression follows by repeatedly applying this starting from
the sequence $1$.
\end{proof}

\begin{proof}[Proof of Theorem \ref{log}]
Note first that $\trans(\TT_b)=(b-1)/b$ (by \eqref{p0}, for
example).  The required bounds on $E_n(\TT_b,\mathbf{k})$
follow from Proposition \ref{regular} by routine computations,
which we summarize below.

Observe that restricting to positions $n$ just before a $1$ in
the escape sequence (i.e.\ such that
$e_{n+1}(\TT_b,\mathbf{k})=1$) does not change the $\liminf$ or
$\limsup$ of $\Delta_n$. Let $D_m$ denote the sum of the
base-$b$ digits of $m$.

For the case $k=0$, we have
\[
S_m := \sum_{i=1}^m z_i = \frac{m-D_m}{b-1}.
\]
Therefore if $n+1$ is the position of the $(m+1)$st $1$ in
$e(\TT_b,\down)$, then
$$E_n=m;  \qquad  n=S_m+m;  \qquad  E_n-\trans n=\frac{D_m}{b}.$$
We deduce that for any integer $s\geq 0$, any $b^{s-1}\leq
m<b^s$, and $n$ as above,
\[
\frac 1b \leq E_n-\trans n \leq \frac{b-1}{b} \,s
\]
where the lower and upper bounds are attained by taking
$m=b^{s-1}$ and $m=b^s-1$ respectively. For $m$ in this range
we have $\log_b n = \log_b(m+S_m) = s + O(1)$ as $m\to\infty$.
The claimed expressions for $\liminf_{n\to\infty} \Delta_n$ and
$\limsup_{n\to\infty} \Delta_n$ follow.

Now consider the case $k\geq 1$.  We have
\[
\sum_{i=1}^{m} \ind[\ell_i=b-k]
= \frac{m-D_m}{b-1} + \sum_{j=0}^\infty \ind[m_j\geq b-k],
\]
where $\cdots m_2 \, m_1 \, m_0$ is the base-$b$ expansion of
$m$.  Similarly to the previous case, we deduce that
$$-\frac{k-1}{b} \,s \leq E_n-\trans n \leq \frac{b-k-1}{b}\, s,$$
for $n$ satisfying $E_n=m$ where $m$ has $s$ base-$b$ digits,
with the bounds being attained when $m$ has all digits $b-k$,
and all digits $b-k-1$, respectively.  The claimed expressions
follow.
\end{proof}

\section{Brushes}
\label{s:brush}

In preparation for the proof of Theorem \ref{T:many_escape}, we next provide
a family of recurrent trees (which are interesting in their own right) with
$E_n(T,\down)$ of order $n^{1-\epsilon}$, for arbitrary $\epsilon>0$.

The \df{$d$-brush} ${\BB}^d$ is a tree defined as follows. The
1-brush is a singly infinite path. For $d>1$, the $d$-brush is
a singly infinite path with a $(d-1)$-brush attached, via its
root, at each non-root vertex.  Note that a $d$-brush has two
principal branches: a $d$-brush and a $(d-1)$-brush. (Another
interpretation is as follows.  Index the non-root vertices of
$\TT_2$ by binary sequences, with the base having the empty
sequence and the children of a vertex having the vertex's label
with $0$ or $1$ appended. With this labeling, the $d$-brush is
a subgraph of $\TT_2$ that consists of all vertices with
Hamming weight less than $d$.  It is also easy to embed $\BB^d$
in $\Z^d$.) Since we will be concerned only with the
configuration $\down$, the orders of children do not matter by
Proposition \ref{rot-mech}.

\begin{prop}[Brushes]\label{brush}
The $d$-brush has $E_n(\BB^d,\down) \sim c n^{\beta}$ as $n\to\infty$, for
$\beta = 1-2^{1-d}$ and some $c=c(d)\in(0,\infty)$.
\end{prop}

The proof uses the following lemmas. The constant $c(d)$ is
easily computable (recursively) from the lemmas, but this will
not be needed.

\begin{lemma}\label{bruch-rec}
  Let $E_n^d=E_n(\BB^d,\down)$ be the escape number for the
  $d$-brush. For fixed $d\geq 2$, define the sequence $(s_i)=(s_i^d)_{i\geq
    1}$ by $s_1=1$ and
  \[
  s_{i+1}-s_i=E^{d-1}_{s_1+\cdots +s_i}, \quad i\geq 1.
  \]
  Then
  \begin{equation}\label{recur}
    E^d_{s_1+\cdots +s_i}=s_i, \quad i\geq 1.
  \end{equation}
\end{lemma}

For example, since $E_n^1=1$ for all $n>0$, we have $s_i^2=i$,
so \eqref{recur} gives $E^2_{i(i+1)/2} = i$.

\begin{proof}Fix $d\geq 2$.
Writing $e^d=e(\BB^d,\down)$, by the explosion formula, Theorem
\ref{explosion}(i), we have
  \[
  e^{d}=\ex(e^{d}+e^{d-1}),
  \]
  and therefore for all $n\geq 1$,
  \begin{equation}\label{ex-brush}
    E^d_{E^d_n+E^{d-1}_n+n}=E^d_n+E^{d-1}_n.
  \end{equation}
  We prove \eqref{recur} by induction on $i$. It holds for $i=1$ since the
  first particle escapes. Assuming it holds for $i$, and taking
  $n=s_1+\cdots+s_i$, we have
  \[
  E_n^d+E_n^{d-1}=s_i+(s_{i+1}-s_i)=s_{i+1},
  \]
  so equation \eqref{ex-brush} becomes
  \[
  E^d_{n + s_{i+1}} = s_{i+1}.  \qedhere
  \]
\end{proof}

\begin{lemma}\label{L:asymp}
  Suppose $f:\N\to\N$ satisfies $f(n)\sim a n^\alpha$ as $n\to\infty$, with
  $a\in(0,\infty)$ and $\alpha\in(0,1)$. Let $s_1 = 1$, and inductively
  $s_{m+1} = s_m + f(s_1+\dots+s_m)$ for $m\geq 1$. If $g:\N\to\N$ is
  increasing and satisfies $g(s_1+\dots+s_m) = s_m$ for all $m$, then
  $g(n)\sim b n^{\beta}$, where $\beta=(1+\alpha)/2$, and $b =
  \sqrt{\frac{2a}{1+\alpha}} \left(\frac{1-\alpha}{2}\right)^{1+\alpha}$.
\end{lemma}

\begin{proof}
  Write $S_m = \sum_{i=1}^m s_i$. We will show that $s_m\sim d m^\gamma$ with
  \[
  \gamma = \frac{1+\alpha}{1-\alpha}
  \qquad\text{and}\qquad
  d = \left(\frac{a(1-\alpha)^{1+\alpha}}{2^\alpha(1+\alpha)}
  \right)^{1/(1-\alpha)}.
  \]
  This implies $S_m \sim \frac{d}{1+\gamma} m^{1+\gamma}$, and the lemma
  then follows by the given properties of $g$, since $\gamma/(1+\gamma) =
  \beta$.

  Let $L = \limsup_{m\to\infty} s_m /m^\gamma$. We first show $L < \infty$.
  Since $S_m\to\infty$, for some $m_0$, for all $m\geq m_0$ we have $f(S_m) <
  2a S_m^\alpha$. Let $C \geq \max_{m\leq m_0} s_m/m^\gamma$ be such
  that for all $m$
  \[
  C m^\gamma + C^\alpha \frac{2a}{(1+\gamma)^\alpha} (m+1)^{\gamma-1}
  < C(m+1)^\gamma
  \]
  (it is easy to see such $C$ exists since $\alpha<1$).
  We now prove by induction that $s_m < C m^\gamma$ for all $m$ (hence
  $L\leq C$). This is known for $m\leq m_0$. If it holds up to $m$, then
  $S_m \leq \frac{C}{1+\gamma} (m+1)^{1+\gamma}$, so
  \begin{align*}
    s_{m+1} = s_m + f(S_m)
    &\leq C m^\gamma + 2a \left(\frac{C}{1+\gamma}\right)^\alpha
    (m+1)^{\alpha(1+\gamma)} \\
    &= C m^\gamma + C^\alpha \frac{2a}{(1+\gamma)^\alpha} (m+1)^{\gamma-1} \\
    &< C (m+1)^\gamma,
  \end{align*}
completing the induction.

In a similar manner, integrating, using $f(n)\sim a n^\alpha$, and
integrating again gives the sequence of inequalities
  \begin{align*}
    \limsup \frac{S_m}{m^{1+\gamma}} &\leq \frac{L}{\gamma+1}; \\
    \limsup \frac{s_{m+1}-s_m}{m^{\gamma-1}} &\leq
    a \left(\frac{L}{\gamma+1} \right)^\alpha; \\
    L = \limsup \frac{s_m}{m^\gamma} &\leq
    \frac{a}{\gamma} \left(\frac{L}{\gamma+1} \right)^\alpha.
  \end{align*}
  Since $\alpha<1$, solving for $L$ gives the upper bound
\begin{equation}\label{Lbound}
  L \leq\left(\frac{a}{\gamma(\gamma+1)^\alpha} \right)^{1/(1-\alpha)}.
\end{equation}

  Similarly, let $\ell = \liminf_{m\to\infty} s_m/m^\gamma$. We first show
  $\ell>0$. For $m>m_1$ we have $f(S_m) > a S_m^\alpha/2$. Take
  $0<c<\min_{m\leq m_0} s_m/m^\gamma$, small enough that for all $m$
  \[
  c m^\gamma + c^\alpha \frac{a}{2(1+\gamma)^\alpha} m^{\gamma-1} >
  (m+1)^\gamma.
  \]
  Then $S_m \geq \frac{c}{1+\gamma} m^{\gamma-1}$, and so
  \begin{align*}
  s_{m+1} = s_m + f(S_m)
  &\geq c m^\gamma + \frac{a}{2} \left(\frac{c}{1+\gamma}\right)^\alpha
  m^{\gamma-1} \\
  &= c m^\gamma + c^\alpha \frac{a}{2(1+\gamma)^\alpha} m^{\gamma-1} \\
  &> c (m+1)^\gamma.
  \end{align*}
By induction we find $s_m>c m^\gamma$ for all $m$.

By the same argument as for the $\limsup$, but with the inequalities
reversed, this implies
$$\ell \geq \frac{a}{\gamma} \left(\frac{\ell}{\gamma+1} \right)^\alpha.$$
Solving shows that $\ell$ is bounded below by the same quantity as in
\eqref{Lbound}, hence $\ell=L$.
\end{proof}

\begin{proof}[Proof of Proposition \ref{brush}]
For the $1$-brush we have $E_n(\BB^1,\down) = 1$ for all $n$.
For other $d$ the result follows by induction from
Lemmas~\ref{bruch-rec} and~\ref{L:asymp}.
\end{proof}

\section{Large-discrepancy trees}
\label{s:disc}

In this section we prove Theorem \ref{T:many_escape}.  We begin with a
corollary of Proposition \ref{brush}.

\begin{cor}\label{hydra}
For any $\epsilon>0$ there exists a tree $T$ of maximum degree $3$ which is
recurrent for random walk and satisfies $E_n(T,\down) > n^{1-\epsilon}/2$.
\end{cor}

\begin{proof}
For $k,d\geq 1$, let $T(k,d)$ be a tree defined as follows:
start with $(\TT_2)^k$, a binary tree of depth $k$, and attach
$2$ disjoint copies of the brush $\BB^d$, via their roots, to
each of its $2^{k-1}$ leaves.  This is recurrent since it has
only countably many ends.  Take $d$ sufficiently large that
$2^{1-d}<\epsilon$. We claim that $T(k,d)$ satisfies the
claimed bound for $k$ large enough.  In what follows we
suppress the rotor configuration in the notation $E_n$, since
it will always be $\down$.

First note that if $H$ is any tree satisfying $E_n(H)\sim a
n^\alpha$, with $\alpha<1$, then the tree $H'$ having two
disjoint copies of $H$ as its principal branches satisfies
$E_n(H')\sim 2 a n^\alpha$, by Theorem \ref{explosion}(i).
Using Proposition \ref{brush} and applying this repeatedly
shows that for any $d$, for $k$ sufficiently large we have
$E_n(T(k,d))\sim n^\beta$ with $\beta=1-2^{1-d}$, and so for
$n$ sufficiently large (say $n>N=N(k(d))$), we have
$E_n(T(k,d))\geq n^\beta/2$.

Now for any $k'$ consider $E_n(T(k',1))$ (which is a finite binary tree with
a singly infinite path attached at each leaf).  We claim that
$E_n(T(k',1))\geq n/2\geq n^\beta/2$ for all $n\leq 2^{k'}$.  Indeed, the
escapes in $T(k',1)$ occur at precisely the times of the first $2^{k'}$
escapes in the infinite binary tree $\TT_2$, as may be seen from the proof of
Proposition \ref{regular} (case $k=0$), for example.  And the claimed
inequality then follows from Theorem \ref{multi}(ii).

Now take $k'$ large enough that $2^{k'}>N$ and $k'\geq k$, and observe that
$T(k',d)$ contains (isomorphic copies of) both $T(k,d)$ and $T(k',1)$ as
subgraphs with the same root.  Therefore by Proposition \ref{subgraph}, for
all $n$,
\[E_n(T(k',d))\geq\max\Bigl[E_n(T(k,d)),E_n(T(k',1))\Bigr]\geq
n^\beta/2.\qedhere\]
\end{proof}

\begin{proof}[Proof of Theorem~\ref{T:many_escape}]
Let $(\epsilon_k)_{k\geq 1}$ be a $(0,1]$-valued sequence to be determined
later.  Construct a tree $T$ by taking an infinite path with vertices
$\rt,x_1,x_2,\ldots$, and attaching a recurrent tree $T(k)$ that satisfies
$E_n(T(k),\down) > n^{1-\epsilon_k}/2 \;\forall n$, via its root, to vertex
$x_k$, for each $k\geq 1$.  Such trees $T(k)$ exist by Corollary~\ref{hydra},
and the resulting tree $T$ is clearly recurrent.  We will choose the
$\epsilon_k$ so that it satisfies the required bound.  In the following the
rotor configuration is always $\down$.

First, if $H$ is any infinite tree, let $H^{(k)}$ be the tree
formed by attaching a path of $k$ edges to the root of $H$, and
taking the other end of this path as the new root.  If $e(H)=1
\,0^{a_1}\, 1\, 0^{a_2}\, 1 \cdots$ then by
Theorem~\ref{explosion}, $e(H^{(k)})=1 \,0^{a_1+k}\, 1\,
0^{a_2+k}\, 1 \cdots\preceq 1 \,0^{ka_1+k}\, 1\, 0^{ka_2+k}\, 1
\cdots$, so we have $E_n(H^{(k)})\geq E_{\lceil n/k\rceil}(H)$
for all $n$.

Since the tree $T$ constructed in the first paragraph contains each
$T(k)^{(k)}$ as a subgraph, by Proposition \ref{subgraph} we deduce that it
satisfies
\begin{equation}\label{disc-bound}
E_n(T) \geq \frac12 \left(\frac{n}{k}\right)^{1-\eps_k}
\geq \frac{n}{2 k n^{\eps_k}}
\end{equation}
for all $n,k\geq 1$.

Now suppose we are given $f$ with $f(n)=o(n)$.  Since the desired statement
involves only sufficiently large $n$ we may assume without loss of
generality that
$f(n)< n/4$ for all $n$.  For each $n$, let $k(n) = \bigl\lfloor \tfrac12
\sqrt{n/f(n)} \bigr\rfloor$, and note that this is at least $1$ by the last
assumption. Let
\[
\epsilon_k := 1 \wedge \min \Bigl\{\log_n \sqrt{n/f(n)}: k(n)=k\Bigr\}.
\]
Since $k(n)\to\infty$ as $n\to\infty$, this is the minimum of a finite set,
and $f(n)<n$, so $\epsilon_k>0$.  This choice ensures that for all $n$ we
have $n^{\epsilon_{k(n)}} \leq \sqrt{n/f(n)}$, therefore taking $k=k(n)$ in
\eqref{disc-bound} gives $E_n(T)\geq f(n)$ for all $n$.
\end{proof}

\section{Random configurations}
\label{s:random}

In this section we prove Theorem \ref{phase}(i).  We begin with an exact
result on finite regular trees, which is somewhat remarkable in that it holds
regardless of the initial rotor configuration.

\begin{prop}[Finite regular trees]\label{exact}
Consider $\TT_b^{h+1}$, the $b$-ary tree truncated at level $h+1$, with an
arbitrary initial rotor configuration.  Let $S$ be the set of the $b^h$
leaves at level $h+1$.  If $1+b+\cdots+b^h$ rotor particles are started in
succession at $\rt$, and stopped when they enter $S\cup\{\rt\}$, then exactly
one particle is absorbed at each vertex of $S$, and the rest are absorbed at
$\rt$.
\end{prop}

\begin{proof}
Write $t(h)=1+b+\cdots +b^h$. We use induction on $h$.  The
result clearly holds when $h=1$.  Assume it holds for $h-1$.
Let $r$ be the initial rotor configuration.

We first apply the explosion formula to bound below the total
number of particles escaping to $S$.  By the inductive
hypothesis, each of the $b$ principal branches $T_i$ has
$E_{t(h-1)}(T_i,S_i,r_i)=b^{h-1}$.  Therefore, in the worst
case $r(\nx)=b$, the sequence $z:=\sum_{i=1}^b \theta e(T_i)$
satisfies $\sum_{j=1}^{1+t(h-1)} z_j = b^h$.  By Theorem
\ref{explosion}(ii), this implies $E_n(T)\geq b^h$ for
$n=1+t(h-1)+b^h-1=t(h)$, i.e.\ at least $b^h$ of the $t(h)$
particles are absorbed by $S$.

Now suppose that some leaf $w\in S$ absorbs no particles.  Then
by the inductive hypothesis, the principal branch containing
$w$ must have received a particle from $\nx$ strictly fewer
than $t(h-1)$ times, hence (by considering the rotor at $\nx$),
no principal branch received a particle strictly more than
$t(h-1)$ times.  Therefore by the inductive hypothesis again,
no leaf of $S$ absorbed more than one particle, which
contradicts the above bound on the total absorptions by $S$.
\end{proof}

\begin{cor}\label{nonzero}
Let $R$ be an i.i.d.\ random rotor configuration on $\TT_b$ satisfying $\E
R(v)<b-1$.  There exist $\delta,c,C>0$ such that for all $n$,
$$\P\big(E_n(\TT_b,R) < \delta n\big)\leq C e^{-cn}.$$
\end{cor}

The proof of the above result will use a version of the Abelian
property. This requires the concept of simultaneous rotor
walks, which we describe next (see e.g.\ \cite{hlmppw} for more
details).  We will give the property we need for finite trees,
and then use Lemma~\ref{trunc}.  (An alternative approach would
be to formulate a version of the Abelian property for infinite
graphs.)

Let $T$ be a finite tree with root $\rt$, and $S\not\ni \rt$ a non-empty set
of its leaves, as in Section \ref{s:trunc}.  At each vertex there is a rotor,
as usual, and at each vertex there is some non-negative integer number of
particles.  At each step of the process, we choose a vertex $v\not\in
S\cup\{\rt\}$ at which there is at least one particle (if such exists), and
\df{fire} $v$; that is, increment the rotor at $v$, and move one particle
from $v$ in the new rotor direction.

\begin{lemma}[Abelian property]\label{abel}
Let $T$ be a finite tree with set of sinks $S$ and initial
rotor configuration $r$.  If we start with $n$ particles at the
base $\nx$ and perform any legal sequence of firings until all
particles are in $S\cup\{\rt\}$, then the process terminates in
a finite number of steps, and the number of particles in $S$ is
exactly $E_n(T,S,r)$.
\end{lemma}

We remark that even stronger statements hold: the number of
times each edge is traversed, and in particular the final
distribution of the particles in $S$ and the final rotor
configuration also do not depend on the sequence of firings.

\begin{proof}
We will adapt the Abelian property as formulated in
\cite{hlmppw}, for which we need to modify the graph slightly.
Assume $\nx\not\in S$ (otherwise the result is trivial).  For
each edge of $T$ not incident to $S\cup\{\rt\}$, replace it
with two directed edges, one in each direction. For each edge
incident to $S\cup\{\rt\}$, replace it with a single edge
directed towards $S\cup\{\rt\}$. Finally, add an extra `global
sink' vertex $\Delta$ and add a directed edge from each vertex
in $S\cup\{\rt\}$ to $\Delta$.  By \cite[Lemma~3.9]{hlmppw},
starting from a given configuration of particles and rotors, in
any legal sequence of firings on this directed graph (stopping
when all particles are at $\Delta$), each vertex fires the same
number of times. In particular, if $n$ particles start at
$\nx$, this means that the number of particles that reach $S$
does not depend on the firing sequence (because it is the
number of firings in $S$). For the firing sequence in which we
move only one particle until it reaches $\Delta$, then move the
next particle, and so on, this number is $E_n(T,S,r)$. (The
initial transitions from $\rt$ to $\nx$ in our usual
formulation of the rotor walk clearly do not matter).
\end{proof}

\begin{proof}[Proof of Corollary \ref{nonzero}]
It suffices to prove the claimed result with $n$ of the form $1+b+\cdots+b^h$
(for $h$ an integer), since this implies the bound (with different constants)
for all $n$.

Since $\E R(v)<b-1$, the branching process from the proof of Theorem
\ref{phase}(ii) is supercritical, and hence with positive probability there
is a live path starting at $\nx$, implying that the first particle escapes.
Thus
$$p:=\P(E_1(\TT_b,R)=1)>0.$$
Let $X$ be the set of vertices $v$ at level $h+1$ for which the subtree $T_v$
based at $v$ satisfies $E_1(T_v,R_v)=1$.  We claim that for
$n:=1+b+\cdots+b^h$ we have
\begin{equation}\label{binomial}
E_n(\TT_b,R)\geq \# X.
\end{equation}
Once this is proved, the result follows from a standard large-deviation bound
(e.g.\ \cite[Theorem 9.5]{durrett2}), since $\# X$ has
$\text{Binomial}(b^h,p)$ distribution.

To prove \eqref{binomial}, by Lemma~\ref{trunc} it suffices to prove that for
all $H>h$,
\begin{equation}\label{binomialH}
E_n(\TT_b^H,S^H,R^H)\geq \# X.
\end{equation}
In the finite tree $\TT_b^H$, start $n$ particles at $\nx$,
stopping them when they enter $S^{h+1}\cup\{\rt\}$.  By
Proposition \ref{exact}, this results in one particle at each
vertex of $S^{h+1}$ and the rest at $\rt$, and the rotors at
distance $h+1$ or greater from $\rt$ are not affected.  Now for
a vertex $v\in X$, continue the rotor walk on $\TT_b^H$ for the
particle located there until it enters $S^H\cup \{\rt\}$; since
in the infinite tree the particle would escape without leaving
the subtree $T_v$, the same applies in the truncated tree.
Therefore, allowing each of the particles at elements of $X$ to
continue in this way results in all of them reaching $S^H$.
Finally we can continue the rotor walks for the remaining
particles (those at the other vertices of $S^{h+1}$) until they
reach $S^H\cup\{\rt\}$.  By Lemma \ref{abel}, \eqref{binomialH}
follows.
\end{proof}

\begin{lemma}[Liminf recursion]\label{liminf}
Fix any tree $T$ and rotor configuration $r$, and let $T_1,\ldots,T_b$ be the
principal branches.  Let
$$\ell:=\liminf_{n\to\infty} \frac{E_n(T,r)}{n};\quad
\ell_i:=\liminf_{n\to\infty} \frac{E_n(T_i,r_i)}{n}, \quad i=1,\ldots ,b;$$
then
$$\ell\geq 1-\frac{1}{1+\sum_{i=1}^b \ell_i}.$$
\end{lemma}

\begin{proof}
Write $E_n=E_n(T,r)$ and $E^i_n=E_n(T_i,r_i)$.  An application of the
explosion formula, Theorem \ref{explosion}(i), shows that $E_{m+n}=m$ for
some $m=m_n=c_n+\sum_{i=1}^b E^i_n$ with $c_n\in[0,b]$ (the error term $c_n$
comes from the shifts $\theta$ in the explosion formula).  Hence
\[
\frac{E_{m+n}}{m+n} = 1 - \frac{1}{1+m/n}
= 1 - \frac{1}{1 + c_n/n + \sum_i (E^i_n/n)}.
\]
Since $\liminf_n\sum_i (E^i_n/n)\geq\sum_i \ell_i$, and $E_n$ is monotone in
$n$, the result follows.
\end{proof}

\begin{proof}[Proof of Theorem \ref{phase}(i)]
Let $L=\liminf_{n\to\infty} E_n(\TT_b,R)/n$.  By Corollary~\ref{nonzero} and
the Borel-Cantelli lemma we have $L\geq \delta$ a.s., where $\delta>0$.  On
the other hand, by Lemma~\ref{liminf}, if for some constant $a>0$ we have
$L\geq a$ a.s., then $L\geq 1-1/(1+ba)$ a.s.  Applying this repeatedly gives
$L\geq 1-b^{-1}=\trans$ a.s., since this is the fixed point of the iteration
$a\mapsto 1-1/(1+ba)$.  Now use Theorem \ref{oded}.
\end{proof}

\section{Rotors about to point away from the root}
\label{s:b-1}

In this section we prove Theorem \ref{b-1}. We say that a rotor
configuration $r$ is \df{$b$-free} if for every $v\in V_\rt$ we
have $r(v)\neq b(v)$; i.e.\ every rotor will send its next
particle away from the root.  We note in particular that this
implies $b(v)\geq 1$ for every $v$, and hence every subtree
$T_v$ is infinite. (In any case, trees in which some vertices
have no children are of little interest to us, as remarked in
the introduction).

\begin{lemma}\label{half}
For any $b$-free rotor configuration $r$ on a tree $T$, we have for all
$n\geq 1$,
$$E_n(T,r)\geq \trans n/2.$$
\end{lemma}

\begin{proof}
We proceed by induction on $n$.  It is easy to see that the first particle
escapes, so the inequality holds for $n=1$.  Now take $n>1$ and assume that
it holds for all smaller $n$, for all $T$ and $b$-free $r$.

Let $T_1,\ldots,T_b$ be the principal branches of $T$, and write
$\trans_i=\trans(T_i)$ and $\trans=\trans(T)$. Let $n_i$ be the total number
of particles that enter $T_i$ when $n$ rotor walks are run from $\rt$.  Since
$n-E_n(T)$ is the number that return to $\rt$, we have
$$n_i\geq n-E_n(T)-1.$$
Also, since $n>1$ we have $n_i<n$, so by the inductive hypothesis,
$$E_{n_i}(T_i) \geq \trans_i n_i/2.$$

Note that $\sum_{i=1}^b \trans_i=\trans/(1-\trans)$ (similarly to
\eqref{p0}). Hence, using the above inequalities, we have
$$E_n(T) =\sum_i E_{n_i}(T_i) \geq \tfrac12 \sum_i \trans_i n_i
\geq \frac{\trans}{2(1-\trans)} \bigl(n-E_n(T)-1\bigr),$$
and solving gives
$$E_n(T)\geq \frac{\trans}{2-\trans} \, (n-1).$$
If $\trans n >2$, some algebra shows that the last quantity is greater than
$\trans n /2$, so the required inequality holds.  Otherwise, $E_n(T)\geq 1
\geq \trans n/2$.
\end{proof}

\begin{proof}[Proof of Theorem \ref{b-1}]
By Proposition \ref{lip}, we may assume that $r$ is $b$-free.  By Theorem
\ref{oded}, it suffices to prove that $\liminf_{n\to\infty} E_n/n \geq
\trans$.

For any vertex $v\in V_\rt$, let $\trans_v := \trans(T_v)$, which is the
probability that a simple random walk starting at $v$ never visits $v$'s
parent. Also for any $v\in V$, let $h(v)$ be the probability that a random
walk started at $v$ ever hits $\rt$. Note that $h$ is harmonic except at the
root, where $h(\rt)=1$, and that
\[
h(v) = \prod_{x\in(\rt,v]} (1-\trans_x).
\]
where $(\rt,v]$ denotes the path from
$\rt$ to $v$ (including $v$ but not $\rt$).  Also let
$$\ell_v:=\liminf_{n\to\infty} \frac{E_n(T_v,r_v)}n.$$
and define by analogy with $h$ the function
\[
f(v) := \prod_{x\in(\rt,v]} (1-\ell_x).
\]
We need to prove $\ell_\nx\geq \trans_\nx$, which is equivalent to
$f(\nx)\leq h(\nx)$.  (In fact the argument will imply $f\equiv h$).

We first show that $f$ is sub-harmonic except at $\rt$.  By
Lemma~\ref{liminf} we have for any $v\in V_\rt$,
\begin{align*}
\ell_v \geq 1-\frac{1}{1+L},  \quad\text{where}\quad
L = \sum_{i=1}^{b(v)} \ell_{v^{(i)}}.
\end{align*}
(The $\trans_v$'s satisfy the corresponding {\em equality}, which is
equivalent to the harmonicity of $h$). Therefore, using the definition of
$f$,
\[
  f(v^{(0)}) = \frac{f(v)}{1-\ell_v}
  \geq f(v)(1 + L ),
\]
and so
\begin{align*}
\sum_{i=0}^{b(v)} f(v^{(i)}) &\geq
f(v)(1+L) + \sum_{i=1}^{b(v)} f(v)(1-\ell_{v^{(i)}}) \\
&= f(v)(1+b(v)),
\end{align*}
which is the claimed sub-harmonicity.

Let $\partial T$ be the set of ends of the tree $T$.  Since $h$ and $f$ are
decreasing along any end, we extend their definitions to $\partial T$ via
limits. We claim that
\begin{equation}\label{implies}
h(\eta)=0 \implies f(\eta)=0, \quad \eta\in\partial T.
\end{equation}
To prove this, suppose $h(\eta)=0$. Thus $\prod_{v\in\eta} (1-\trans_v)=0$,
and so $\sum_{v\in\eta} \trans_v = \infty$.  Since by Lemma~\ref{half} we
have $\ell_v\geq \trans_v/2$ for all $v$, this implies $\sum_{v\in\eta}
\ell_v = \infty$, and thus $f(\eta)=0$.

Now let $(X_t)$ be the simple symmetric walk on $T$, started at $\nx$, and
stopped at the first visit (if any) to $\rt$.  Note that $(X_t)$ almost
surely either hits $\rt$ or visits exactly one end $\eta$ infinitely often --
in that case we say the walk escapes to $\eta$. Let $\mu$ be the associated
harmonic measure on $\{\rt\}\cup\partial T$, so $\mu(A)$ equals the
probability that the
walk escapes to some end in $A$. Now $(h(X_t))$ is a bounded martingale,
therefore it converges, and $\E h(X_0) = \E \lim_{t\to\infty} h(X_t)$; that
is,
\begin{align*}
h(\nx)&=
\mu(\{\rt\}) \cdot h(\rt) + \int_{\partial T} h \;d\mu \\
&= h(\nx)\cdot 1 +\int_{\partial T} h \;d\mu.
\end{align*}
Thus, the last integral is $0$, so $\mu\{\eta\in\partial T: h(\eta)> 0\}=0$.
(This fact also follows from \cite[Proposition~8]{benjamini-peres}, for
example). By \eqref{implies}, this implies $\mu\{\eta\in\partial T: f(\eta)>
0\}=0$. On the other hand, $(f(X_t))$ is a bounded sub-martingale, so we
obtain similarly
\begin{align*}
f(\nx)&\leq
 h(\nx)\cdot 1+ \int_{\partial T} f \;d\mu\\
 &=h(\nx),
\end{align*}
completing the proof.
\end{proof}

\section{Maximum depth}
\label{s:depth}

\begin{proof}[Proof of Theorem~\ref{T:depth}(i)]
By Theorem~\ref{phase}, if $\E R(v)\geq b-1$ then all particles return to the
root a.s.  Let $V_n$ be the set of non-root vertices that have ever been
visited when the $n$th particle returns to the root.  Since a rotor always
points to the last direction in which a particle left, at this time all
vertices in $V_n$ have their rotors pointing towards the root.  It follows
that particle $n+1$ will visit all vertices in $V_n$, as well as all their
children.  Let $\Delta V_n$ be the set of children of $V_n$ that are not
themselves elements of $V_n$, and note that $\#\Delta V_n =(b-1)\#V_n+1$.
Each time particle $n+1$ enters a new element of $\Delta V_n$, it encounters
a previously untouched subtree whose rotors are distributed as in the
original tree. So before returning to $V_n$, it visits a number of new
vertices that is equal in law to $\#V_1$. Thus
\begin{equation}\label{explored}
\# V_{n+1} = \# V_n + \sum_{i=1}^{(b-1)\#V_n+1} Z_i,
\end{equation}
where $(Z_i)$ are i.i.d.\ with the same law as $\# V_1$ and
independent of $V_n$.  (Thus, $(\# V_n)_{n\geq 1}$ is `almost'
a Galton-Watson process).

On the other hand, if a vertex $v$ belongs to $V_1$, then so do all those of
its children $v^{(k)}$ that satisfy $k>R(v)$ (and no others).  Thus the graph
induced by $V_1$ is a Galton-Watson tree with offspring distribution that of
$b-R(v)$.

We write $c_i,C_i$ for (small, large) constants in $(0,\infty)$
depending only on $b$ and the distribution of $R(v)$.   Since
$\E R(v)=b-1$, the latter Galton-Watson tree is critical with
bounded offspring distribution, therefore (see e.g.\ \cite[\S
2.1: Theorem 2 and Lemma 4]{kolchin}) for all $N\geq 1$,
\begin{gather}
c_1/N \leq \P(\#V_1 >N^2)\leq C_1/N;\label{GWvolume}\\
c_2/N \leq \P(D_1 >N)\leq C_2/N\label{GWdepth}.
\end{gather}

By \eqref{GWvolume}, for $(Z_i)$ i.i.d.\ with the same law as $\# V_1$,
$$\P\Bigl(\sum_1^N Z_i > N^2\Bigr)\geq \P\bigl(\max_1^N Z_i > N^2\bigr)
\geq (1-c_1/N)^N\xrightarrow{N\to\infty} e^{-c_1},$$ therefore
$$\P\Bigl(\sum_1^N Z_i > N^2\Bigr)\geq c_3, \quad\forall N\geq 1.$$
(For large enough $N$ this follows from the previous line,
while the constant can be chosen so that it holds for small $N$
because $Z_1$ has unbounded support). Since $(b-1)N +1\geq N$,
it follows from \eqref{explored} that
$$\P\Bigl(\# V_{n+1} > (\# V_n)^2\mid V_1,\ldots,V_n\Bigr)>c_3.$$
We have in any case $\#V_{n+1}> \#V_n$ and $\#V_2\geq 2$, so by the law of
large numbers,
\begin{equation}\label{volume}
\P\bigl(\#V_n \geq 2^{2^{(c_4n)}}\bigr)\to 1 \quad\text{as }n\to\infty.
\end{equation}

We now use a similar argument to convert \eqref{volume} to a lower bound on
the depth. Recall that $D_n$ is the maximum level of all vertices in $V_n$,
and note that
$$D_{n+1}\geq \max_{i=1}^{(b-1)\#V_n+1} Y_i,$$
where $(Y_i)$ are the depths of the subtrees added to $V_n$ to form
$V_{n+1}$, which are i.i.d.\ with the same law as $D_1$. By \eqref{GWdepth},
$\P(\max_i^N Y_i >\sqrt N)\to 1$ as $N\to\infty$, so $\P\bigl(D_{n+1} \geq
\sqrt{\# V_n}\bigr)\to 1$ as $n\to\infty$, which together with \eqref{volume}
gives $\P\bigl(D_n \geq e^{e^{cn}}\bigr)\to 1$.

Turning to the upper bound, by \eqref{GWvolume} we have for all $t>0$ and
$N\geq 1$,
$$\P\biggl(\frac{\log\sum_1^N Z_i}{\log N}>t\biggr)
= \P\Bigl(\sum_1^N Z_i > N^t\Bigr) \leq N\P(Z_1> N^{t-1})\leq C_1
N^{-\frac{t-3}2}.$$ Hence (by the above for $N\geq 2$ and \eqref{GWvolume} for $N=1$),
$$\E\biggl(\frac{\log\sum_1^N Z_i}{\log N}\biggr)\leq C_3.$$
Thus from \eqref{explored}, $\E \log \# V_n \leq (C_4)^n$, so
Markov's inequality gives $\P\bigl(\#V_n > e^{e^{Cn}}\bigr)\to
0$ as $n\to\infty$. Finally, $D_n$ satisfies the same bound,
since $D_n\leq \#V_n$.
\end{proof}

\begin{proof}[Proof of Theorem~\ref{T:depth}(ii)]
The lower bound $n\leq D_n$ clearly holds regardless of the
rotor configuration, so we turn to the upper bound.  As in part
(i), a.s.\ no particle escapes to infinity, and each particle
visits all vertices visited by its predecessor.

Let $x$ be a vertex of $\TT_b$ at level $h+1$, and watch the rotor walk only
while it is on the path $\pi=\pi(x) $ from $\rt$ to $x$. Since whenever the
particle leaves $\pi$ it always returns via the same vertex, its behaviour on
$\pi$ is identical to that of a rotor walk on a path, as determined by the
initial rotor positions on $\pi$.  If $\pi$ has vertices
$\rt=x_0,x_1,\ldots,x_h,x_{h+1}=x$, let $k_i=k_i(x)$ be such that
$x_{i+1}=(x_i)^{(k_i)}$ for $i=1,\ldots,h$.  Then the $n$th particle reaches
$x$ if and only if
$$\sum_{i=1}^h \ind[ R(x_i)\geq k_i(x)] < n.$$

Now let $X$ be a uniformly randomly chosen vertex at level
$h+1$, independent of the initial rotor configuration.  Then
$(k_i(X))_{i=1}^h$ are i.i.d.\ uniformly random on
$\{1,\ldots,b\}$.  If $K$ is uniform on $\{1,\ldots,b\}$ and
independent of $R(v)$ then $\P(R(v)\geq K)=\E R(v)/b=:p$, say.
By the above, it follows that the probability particle $n$
reaches $X$ is the probability that a Binomial$(h,p)$ random
variable $B$ is less than $n$.

Let $Z$ be the number of vertices at level $h+1$ visited by particle $n$.
Then
\begin{equation}\label{randompath}
\P(D_n>h)=\P(Z>0)\leq \E Z = b^h\,\P(B<n).
\end{equation}
On the other hand, for $\alpha\in(0,1)$, by a standard Chernoff bound (i.e.\
$\P(B<\alpha h)\leq s^{-\alpha h}\E s^B$ with
$s=(p^{-1}-1)/(\alpha^{-1}-1)$),
$$\P(B<\alpha h)\leq
\biggl[\Bigl(\frac p\alpha\Bigr)^\alpha \Bigl(\frac {1-p}{1-\alpha}\Bigr)^{1-\alpha}
\biggr]^h.
$$
We have
$$
b \, \Bigl(\frac p\alpha\Bigr)^\alpha \Bigl(\frac {1-p}{1-\alpha}\Bigr)^{1-\alpha}
\xrightarrow{\alpha\to 0} b(1-p)=b-\E R(v) <1.
$$
Therefore we may choose $\alpha>0$ (depending only on $b$ and
$\E R(v)$) so that the left side is less than $1$.  Then from
\eqref{randompath}, $\P(D_{\lfloor\alpha h\rfloor}>h)\to 0$ as
$h\to\infty$, hence $\P(D_n>Cn)\to 0$ as $n\to\infty$, as
required.
\end{proof}

\section*{Open Problems}

\begin{mylist}
\item Characterize the set of possible escape sequences $e(T,\down)$, where
    $T$ varies over all trees (with root of degree $1$ and all degrees finite),
    and $\down$ is the initial configuration in which all rotors point
    towards the root.
\item If $e=e(T,r)$ is an escape sequence for a tree $T$ (with
    some initial rotor configuration $r$),
    must every sequence $e'$ satisfying $e'\leq e$ also be an escape sequence
    for $T$?
\end{mylist}

\bibliographystyle{habbrv}
\bibliography{bib}
\newpage

\end{document}